\documentclass{amsart}

\usepackage{float}
\usepackage{color}
\usepackage{xypic}
\usepackage{enumerate}
\usepackage{tikz}
\usetikzlibrary{calc}

\def\k{\mathfrak k}      
\def\n{\mathfrak n}      
\def\p{\mathfrak p}      

\def\r{\mathfrak r}

\def\c{\mathbf c}

\newtheorem{theorem}{Theorem}[section]
\newtheorem{lemma}{Lemma}
\newtheorem{corollary}{Corollary}
\newtheorem{proposition}{Proposition}
\theoremstyle{conjecture}

\theoremstyle{definition}
\newtheorem*{definition}{Definition}
\theoremstyle{question}
\newtheorem*{question}{Question}
\theoremstyle{questions}

\newtheorem*{notation}{Notation}
\theoremstyle{examples}
\newtheorem*{examples}{Examples}

\theoremstyle{remark}
\newtheorem*{remark}{Remark}
\theoremstyle{remarks}

\theoremstyle{example}

\numberwithin{equation}{section}

\begin{document}

\title{Exponential of the $S^1$ Trace of the Free Field and Verblunsky Coefficients}

\author{Mohammad Javad Latifi}
\email{mjlatifi@math.arizona.edu}

\author{Doug Pickrell}
\email{pickrell@math.arizona.edu}

\begin{abstract} An identity of Szego, and a volume calculation, heuristically suggest
a simple expression for
the distribution of Verblunsky coefficients
with respect to the (normalized) exponential of the $S^1$ trace of the Gaussian free field. This heuristic expression is not quite correct. A proof of the correct formula has been found by Chhaibi and Najnudel (\cite{CN}). Their proof uses random matrix theory and overcomes many difficult technical issues. In addition to presenting the Szego perspective, we show that the Chhaibi and Najnudel theorem implies a family of combinatorial identities (for moments of measures) which are of intrinsic interest.
\end{abstract}
\maketitle

\setcounter{section}{-1}

\section{Introduction}

Let $Prob(S^1)$ denote the metrizable compact convex set of probability measures on
the unit circle $S^1$ with the weak* topology, let $Prob'(S^1)$ denote the subset
consisting of measures with infinite support, and let $\Delta:=\{z\in\mathbb C:|z|<1\}$. The
Verblunsky correspondence refers to a miraculous homeomorphism
\begin{equation}\label{verblunskymap}Prob'(S^1) \leftrightarrow
\prod_{n=1}^{\infty}\Delta: \mu \leftrightarrow \alpha\end{equation}
or more generally to a homeomorphism of $Prob(S^1)$ and a compactification of the product space.
This correspondence can be described in several different ways. One formulation is as follows.
Given $\mu\in Prob'(S^1)$, if $p_0 =z^0$, $p_1(z)$, $p_2(z)$,... are the monic orthogonal polynomials corresponding
to $\mu$, then $\alpha_n =p_{n}(0)^*$ (where $(\cdot)^*$ is complex conjugation). Conversely, given $\alpha$, $\mu$ is the weak* limit of probability measures
\begin{equation}\label{verblunskymap2}d\mu= \lim_{N\to\infty}
\frac{\prod_{n=1}^{N}(1-\vert \alpha_n\vert^2)}{\vert p_N(z)\vert^2}\frac{d\theta}{2\pi}\end{equation}
where $p_0=1$,
\begin{equation}\label{Szego} p_{n}(z) = zp_{n-1}(z)+\alpha_n^*z^{n-1}p_{n-1}^{*}(z),\quad  n >0,\end{equation} and
$p_n^*(z)=p_n(\frac1{z^*})^*$ (our conventions for Verblunsky coefficients differ slightly from those in \cite{Simon2}, see Subsection \ref{notation}). In the text we will use a reformulation of Szego's recursion (\ref{Szego}) which relates Verblunsky coefficients to root subgroup coordinates for the loop group $LSU(1,1)$.

Suppose that $\mu\in Prob(S^1)$ and write
$$\mu=e^f \frac{d\theta}{2\pi}+\mu_{s}, \qquad f(z) \sim\sum_{n=-\infty}^{\infty}f_nz^n,\quad z\in S^1$$
where $\mu_{s}$ is perpendicular to the Lebesgue class. A famous theorem of Szego
asserts that if $\mu_{s}=0$, then
\begin{equation}\label{Szegoidentity}exp(-\sum_{n=1}^{\infty}n\vert f_n\vert^2)=\prod_{n=1}^{\infty}(1-\vert
\alpha_{n}\vert^2)^{n}\end{equation}
In Section \ref{RSFvolume} we will see that for finite $N$,
\begin{equation}\label{volume}\chi_{N}(f_1,...,f_N)\prod_{n=1}^N d\lambda(f_n)= \prod_{n=1}^N(1-|\alpha_n|^2)^{n-1} d\lambda(\alpha_n) \end{equation}
where $\chi_N$ is the characteristic function for the image of the map $(\alpha_1,...,\alpha_N)\to (f_1,...,f_N)$ and $\lambda$ denotes Lebesgue measure. Let
\begin{equation}\label{maindefinition}d\nu_{\beta}:=\prod_{n=1}^{\infty}\frac{n\beta}{\pi}e^{-n\beta\vert
f_n\vert^2}d\lambda(f_n)\end{equation}
At a heuristic level, the identities (\ref{Szegoidentity}) and (\ref{volume}) suggest that
\begin{equation}\label{mainheuristic} \chi_{\infty}(f_1,...) d\nu_{\beta}\stackrel{?}{=}\left(\prod_{n=1}^{\infty}\frac{\beta}{1+\beta}\right)\prod_{n=1}^{\infty}
\frac{nb}{\pi}(1-\vert\alpha_{n}\vert^2)^{nb-1}d\lambda(\alpha_{n})
\end{equation}
where there is a notable shift of exponent $b=1+\beta$, and $\chi_{\infty}$ is a conditioning on the left hand side (which, to give the story away, reduces the statement to $0=0$).

The measures $\nu_{\beta}$, $\beta>0$, can be
characterized in many different ways. For example they exhaust all of the conformally invariant Gaussian distributions for
$f_+\in H^0(\Delta)/\mathbb C$, where $f_+(z)=\sum_{n=1}^{\infty}f_nz^n$. In terms of boundary values, where $f=f_+^*+f_0+f_+$ is
interpreted as a real hyperfunction modulo constants, $\nu_{\beta}$ is the (normalized) $S^1$ trace of the Gaussian free field with inverse temperature $\beta$. From this latter point of view, a notorious complication is that the hyperfunction $f$ is $\nu_{\beta}$-almost surely not an ordinary function on the circle. Consequently the naive map in which a (normalized) real valued function $f$ is mapped to the Lebesgue class
probability measure $e^f\frac{d\theta}{2\pi}$ is defined on a set of $\nu_{\beta}$-measure zero. However, from various points of view, it is well-known how to make sense of the map from $f_+$ (or $f$) to a probability measure in a $\nu_{\beta}$-almost sure sense; we denote this map by $M$.
The following is a corollary of the work of
Chhaibi and Najnudel (\cite{CN}).

\begin{theorem}\label{maintheorem} Suppose that $\beta^2<2$. Then
$$\alpha_*M_*\left(\prod_{n=1}^{\infty}\frac{n\beta}{\pi}e^{-n\beta\vert
f_n\vert^2}d\lambda(f_n)\right)=\prod_{n=1}^{\infty}
\frac{n\beta}{\pi}(1-\vert\alpha_{n}\vert^2)^{n\beta-1}d\lambda(\alpha_{n}) $$
(where $(\cdot)_*$ denotes pushforward of measures).
\end{theorem}

Thus (\ref{mainheuristic}) does correctly predict (at a heuristic level) that Verblunsky coefficients are independent, but the shift in parameter in (\ref{mainheuristic}) is misleading! The restriction on $\beta$ is very likely artificial, but for lack of expertise, we do not see how to dispense with it.

To prove this theorem it would suffice to show that for given $N$, the two measures in Theorem \ref{maintheorem} have the same $(f_1,...,f_N)$ distribution. The $\nu_{\beta}$ distribution is
\begin{equation}\label{distribution}\prod_{n=1}^{N}\frac{n\beta}{\pi}e^{-n\beta\vert
f_n\vert^2}d\lambda(f_n)\end{equation}
Now consider $\mathbf N>N$ and the measure
$$\prod_{n=1}^{\mathbf N}
\frac{n\beta}{\pi}(1-\vert\alpha_{n}\vert^2)^{n\beta-1}d\lambda(\alpha_{n}) $$
Using the Szego and volume identities (\ref{Szegoidentity}) and (\ref{volume}), we can express this in terms of $f_1,...,f_{\mathbf N}$ as
$$\chi_{D_{\mathbf N}}(f_1,...,f_{\mathbf N})\left(\prod_{n=1}^{\mathbf N}\frac{n\beta}{\pi}\right) e^{-\sum_{n=1}^{\infty}n(1+\beta)\vert
f_n(f_1,...,f_{\mathbf N})\vert^2}\prod_{n=1}^{\mathbf N} d\lambda(f_n)$$
(where there is a straightforward but complicated way to express $f_n$ as a function of $f_1,...,f_{\mathbf N}$ for $n>\mathbf N$).
Unfortunately it is not clear, even at a heuristic level, how to integrate out the intermediate variables $f_{N+1},...,f_{\mathbf N}$ and evaluate the resulting distribution in the limit as $\mathbf N\to\infty$. Why is the limit conformally invariant? Why is it Gaussian? This is rather mysterious.

Modulo a number of important details, we will study the inverse of $M$, as realized in the following diagram:
\begin{equation}\label{vdiagram}\begin{matrix} H^0(\Delta,\mathbb C)/\mathbb C & \stackrel{M}{\rightarrow} & Prob(S^1)\\
\updownarrow &  & \updownarrow \\
H^0(\Delta, \mathbb C^{\times})/\mathbb C^{\times} & \leftarrow & \overline{\prod \Delta} \end{matrix} \qquad
\begin{matrix}f_+ & \stackrel{M}{\rightarrow} &\mu\\
\updownarrow & & \updownarrow\\
e^{-f_+}=1+\sum_1^{\infty}x_nz^n& \leftarrow & \alpha \end{matrix}
\end{equation}
(in a deterministic context, one must restrict the domains, see Section \ref{RSF}; in the present probabilistic context, the map $M$ and its inverse are defined in almost sure senses).
The first step in \cite{CN}, depending crucially on random matrix theory, is essentially to show that
\begin{equation}\label{maintheorem2}x_*(\nu_{\beta})
=x_*(\prod_{n=1}^{\infty}\frac{n\beta}{\pi}(1-|\alpha_n|^2)^{n\beta-1}d\lambda(\alpha_n))
\end{equation}
(where slightly abusing notation, on the left, $x:f_+ \to x=e^{-f_+}$, and on the right $x:\alpha \to x$). In terms of moments,
this is equivalent to a set of identities (see Theorem \ref{CNidentity}) which are of intrinsic interest. For example the equality of expected values for $x_nx_n^*$ is equivalent to
\begin{equation}\label{niceidentity}\sum \prod_{u=1}^L\frac{1}{(i(u)\beta+1)(j(u)\beta+1)}=\prod_{k=1}^{n}(\frac{1}{k}\beta^{-1}+\frac{k-1}{k})  \end{equation}
where the sum is over all integral sequences $i(1)>j(1)>...>i(L)>j(L)\ge 0$ satisfying $\sum_{u=1}^L(i(u)-j(u))=n$,
for some $1\le L\le n$. We do not have a conceptual explanation for the appearance of the probability mass function for a harmonic sequence of independent Bernoulli random variables.

At a heuristic level, the moment identities (which have the virtue of being computable), (\ref{maintheorem2}), and
Theorem \ref{maintheorem} are essentially equivalent; consequently the moment identities are well-worth understanding
(see \cite{LP}). However at a technical level, because of the non-continuous nature of the map $M$, there is a gap between (\ref{maintheorem2}) and Theorem \ref{maintheorem}; this is the void which \cite{CN} has filled. It would be highly desirable to understand this using softer techniques.

\subsection{Plan of the Paper}

The first three sections recall standard facts, albeit in slightly idiosyncratic ways. In Section \ref{Conformally Invariant Gaussians} we recall why $\nu_{\beta}$ is the essentially unique conformally invariant Gaussian probability measure on $H^0(\Delta)/\mathbb C$, and why it is necessary to regularize the map $M$. In Section \ref{RSF} we show that Verblunsky coefficients are related to root subgroup coordinates for the loop group of $SU(1,1)$, and we explain the deterministic meaning of the diagram (\ref{vdiagram}). In Section \ref{Mdefinition} we explain how to properly formulate the map $M$. In Section \ref{RSFvolume} we prove the volume formula (\ref{volume}). In Sections \ref{proof1} and \ref{proof2} we calculate moments for the coefficients of $x$ with respect to the $f_+$ and $\alpha$-distributions (in Theorem \ref{maintheorem}), respectively. (\ref{maintheorem2}) is equivalent to the equality of these two calculations, as formulated in Theorem \ref{CNidentity}. In \cite{LP} our goal is to give direct proofs of these identities, but as of this writing, this remains a goal.

\subsection{Notation}\label{notation}

Our conventions regarding Verblunsky coefficients differ from those in \cite{Simon2}, where the nontrivial coefficients are indexed by $n=0,1,...$.
We use the convention that $\alpha_0=1$ and the nontrivial Verblunsky coefficients are $\alpha_1,\alpha_2,...$. The rationale is that the group of rotations of the circle acts naturally on $Prob(S^1)$, and the map to Verblunsky coefficients is equivariant provided that rotation by $\theta$ acts on $\alpha$ (in our notation) by $(\alpha_n) \to (e^{-in\theta}\alpha_n)$. We also switch the sign of the coefficients. Thus our $\alpha_1$ is the negative of the $0$th coefficient in \cite{Simon2}. This eliminates a profusion of signs that would appear otherwise.

\section{Background}\label{Conformally Invariant Gaussians}

The group $PSU(1,1)=\{g=\pm\left(\begin{matrix} a&b\\
\bar {b}&\bar {a}\end{matrix} \right):\vert a\vert^2-\vert b\vert^2
=1\}$ acts
on the open unit disk $\Delta\subset \mathbb C$ by linear fractional
transformations,
\begin{equation}\label{0.1}g:z\to\frac {az+b}{\bar {b}z+\bar {a}}\end{equation}
This identifies $PSU(1,1)$ with the group of all conformal
automorphisms of $\Delta$, or equivalently with the group of
all orientation-preserving isometries of $\Delta$, equipped with
the non-Euclidean arclength $ds=\frac {\vert dz\vert}{1-\vert z\vert^
2}$. For this metric the Gaussian
curvature $=-4$, and the global metric on $\Delta$ is given by
\begin{equation}\label{nonEmetric}d(z,w)=arctanh(|\frac {z-w}{1-\bar {z}w}|)\end{equation}

The action of $PSU(1,1)$ on $\Delta$ induces an action (by pullback) on $H^m(\Delta)$, the space of holomorphic
differentials of order $m$, for $m=0,1,...$. The action
of $g\in PSU(1,1)$ on $H^m$ is
\begin{equation}g:F(z)(dz)^m\to F(\frac {\bar {a}z-b}{-\bar {b}z+a})(-
\bar {b}z+a)^{-2m}(dz)^m\label{0.5}\end{equation}
The space $H^m$ is a Frechet space with respect to the
topology of uniform convergence on compact sets, and the action $(\ref{0.5})$ is continuous.
More generally the universal covering $\widetilde{PSU(1,1)}$ acts on $H^m(\Delta)$ for any $m\ge 0$.

For each $m\ge 0$, the action $\widetilde {PSU(1,1)}\times H^m(\Delta )$ contains an essentially unique
irreducible unitary action (This is proven by considering the (lowest weight) infinitesimal action of the Lie algebra
of $PSU(1,1)$, see e.g. \cite{Lang}). If $m> 1/2$,
then the invariant Hilbert norm is given by an integral
\begin{equation}\label{Hilbertnorm}\vert\vert F(z)(dz)^m\vert\vert^2=\frac {2m-1}{2\pi}\int_{\Delta}
\vert F(z)\vert^2(1-\vert z\vert^2)^{2m-2}dx\wedge dy=\sum_{n\ge 0}\vert F_{n+m}\vert^2B(n+1,2m)\end{equation}
where $F(z)(dz)^m=\sum_{n\ge 0}F_{n+m}z^n (dz)^m$ and
$B(n+1,2m)$ is the Beta function.  We inserted the seemingly
unnatural factor $\frac {2m-1}{2\pi}$
because the last sum shows that this norm
can be analytically continued to $m>0$, since $B(n+1,2m)>0$.

\begin{proposition} Fix $m>0$. For each $\beta>0$, $\nu_{\beta}$ is an ergodic $PSU(1,1)$-invariant Gaussian measure. Conversely, if $\nu$ is an ergodic conformally invariant Gaussian on $H^m(\Delta)$, then $\nu=\nu_{\beta}$ for some $\beta>0$.
\end{proposition}

\begin{proof}The ergodicity of $\nu_{\beta}$ is a special case of a general result of Irving Segal, see \cite{ISegal}.
The second statement follows from the fact that the unitary substructure for the action $PSU(1,1)$ on $H^m(\Delta)$ is essentially unique and irreducible, see e.g. \cite{Lang}. This determines the Cameron-Martin subspace for the Gaussian, hence determines the measure. For a completely different geometric proof, see \cite{Sodin}. \end{proof}

We are primarily interested in the limit $m\downarrow 0$, which is exceptional. The action of $PSU(1,1)$ on $H^0$ is
reducible: $\mathbb C$ is an invariant subspace. The norm
as defined by (\ref{Hilbertnorm}) is not well-defined when $m=0$.
However, in the definition of the norm, we can multiply
by $m$.  In this case we obtain a Hilbert space
substructure for the quotient $H^0/\mathbb C$ (the original vacuum,
the constant $1$, is now a `ghost'), and an isometry induced by the $PSU(1,1)$-equivariant Frechet space isomorphism
\begin{equation}H^0(\Delta)/\mathbb C \stackrel{\partial}{\rightarrow} H(\Delta)^1:f_+\to\Theta :=\partial
f_+.\label{1.9}\end{equation}
To set our notation,
for $\Theta=\sum_{n\ge 0} \theta_{n+1}z^n dz\in H^1(\Delta)$
$$\langle\Theta,\Theta\rangle=\frac{1}{2\pi i}\int \Theta\wedge \bar \Theta=\frac{1}{2\pi}\sum_{n=0}^{\infty}\frac{1}{n+1}|\theta_{n+1}|^2 $$
If $\Theta=\partial f_+$, then $\theta_n=nf_n$, $n=1,2,...$

Define $\nu_{\beta}$ as in (\ref{maindefinition}). If $t=\frac{1}{\beta}$, then $t\to \nu_{\beta}$ is the
unique conformally invariant convolution semigroup of Gaussian measures on the spaces in
(\ref{1.9}). The almost sure properties of the associated random functions are
considered in detail in \cite{K}, which we merely note for motivation.
As random series $\Theta=\sqrt{\beta}\sum (n+1)^{1/2}Z_nz^n$, and
\begin{equation}f_+(z)=\int_0^z\Theta =\sum_1^{\infty}{n}^{-1/2}Z_nz^n.\label{ 1.14}\end{equation}
where $Z_n$ are standard normal complex random variables.
This is the critical case in chapter 13 of \cite{K}.

\begin{proposition}\label{K} (a) Fix an angle $\alpha$.  For all $a\in \mathbb C$, almost surely
$$\liminf_{r\uparrow 1}\vert f_+(re^{i\alpha})-a\vert =0,$$
i.e. for each ray the $f_+$ image is almost surely dense in $\mathbb C$.

(b) Again fix the angle $\alpha$.  Then almost surely
$$f_+(re^{i\alpha})=O(\sqrt {\rho (r)ln(\frac {\rho (\sqrt r)}{1-r}
)})\quad as\quad r\uparrow 1,$$
where $\rho (r)=-ln(1-r^2)$.

\end{proposition}

\begin{question}\label{K2} Is (a) true for any ergodic conformally invariant $\nu \in Prob(H^0(\Delta)/\mathbb C)$, $\nu\ne \delta_0$?

\end{question}

\section{Root Subgroup Coordinates}\label{RSF}

In this section we slightly reformulate the Szego recursion and observe that Verblunsky coefficients are related to so called root subgroup coordinates for loops into $SU(1,1)$. This reformulation is surely not new. In fact the point is that in developing the theory of root subgroup factorization, the second author should have observed early on that some aspects had already been developed in the Verblunsky context.

\subsection{Reformulation of the Szego Recursion}

The Szego recursion can be written as
$$\left(\begin{matrix}p_n\\p_n^*\end{matrix}\right)=\left(\begin{matrix}z&\alpha_n^*z^{n-1}\\\alpha_nz^{-n+1} & z^{-1} \end{matrix}\right)
\left(\begin{matrix}p_{n-1}\\p_{n-1}^*\end{matrix}\right)$$
Consequently there is a closed formula
$$\left(\begin{matrix}p_n\\p_n^*\end{matrix}\right)=\prod_{k=1}^{\stackrel{n}{\leftarrow}}\left(\left(\begin{matrix}1&\alpha_k^*z^{k}\\\alpha_kz^{-k} & 1 \end{matrix}\right)\left(\begin{matrix}z&0\\0 & z^{-1} \end{matrix}\right)\right)
\left(\begin{matrix}1\\1\end{matrix}\right)$$
This does not have any meaning in the limit as $n\to\infty$.

Define the $n$th reversed polynomial to be $r_n(z)=z^np_n^*(z)=z^n\overline{p_n(\frac{1}{\overline z})}$.
Note that $r_n(0)=1$ and (\ref{verblunskymap2}) is equivalent to
\begin{equation}\label{verblunskymap3}\mu= \lim_{N\to\infty}
\frac{\prod_{n=1}^{N}(1-\vert \alpha_n\vert^2)}{\vert r_N(z)\vert^2}\frac{d\theta}{2\pi}\end{equation}
because on the circle $|r_N|=|p_N|$. The point of what follows is that we can take limits for the $r_N$.

\begin{lemma}\label{szegorecursion}In terms of reversed polynomials, the Szego recursion (\ref{Szego}) is equivalent to
$$\left(\begin{matrix}r_n^*\\r_n\end{matrix}\right)=
\left(\begin{matrix}1&\alpha_n^*z^{-n}\\ \alpha_nz^{n} & 1 \end{matrix}\right)\left(\begin{matrix}r_{n-1}^*\\r_{n-1}\end{matrix}\right)
=\prod_{k=1}^{\stackrel{n}{\leftarrow}}\left(\begin{matrix}1&\alpha_k^*z^{-k}\\\alpha_kz^{k} & 1 \end{matrix}\right)
\left(\begin{matrix}1\\1\end{matrix}\right)$$
\end{lemma}

Note that
$$g_2:=\prod_{k=1}^{\stackrel{n}{\leftarrow}}(1-|\alpha_k|^2)^{-1/2}\left(\begin{matrix}1&\alpha_k^*z^{-k}\\ \alpha_kz^{k} & 1 \end{matrix}\right)$$
is a loop with values in $SU(1,1)$. A loop of this form has two properties. First, it has the form
$$g_2= \left(\begin{matrix}d_2^*&c_2^*\\c_2&d_2\end{matrix}\right)$$
where $c_2$ and $d_2$ are polynomials of degree $\le n$ satisfying $c_2(0)=0$ and $d_2(0)=1$. Second, it has a triangular
factorization (essentially a Riemann-Hilbert factorization)
$$g_2(\zeta)=\left(\begin{matrix}1&\bf x^*\\0&1\end{matrix}\right)\left(\begin{matrix}\mathbf a_2&0\\0&\mathbf a_2^{-1}\end{matrix}\right)\left(\begin{matrix}\alpha_2&\beta_2\\\gamma_2&\delta_2\end{matrix}\right) $$
where the first matrix is holomorphic in $\Delta^*$ and $=1$ at $\infty$, $\mathbf a_2$ is a positive constant, and the third matrix is a holomorphic map $\overline{\Delta} \to SL(2,\mathbb C))$ and unipotent upper triangular at $z=0$. These two properties are equivalent (see \cite{CP}).

Instead of $g_2$ (and its root subgroup factorization), we are more interested in the matrix (and its root subgroup factorization) which appears in the Szego recursion
$$\left(\begin{matrix}\delta_2^*&\gamma_2^*\\\gamma_2&\delta_2\end{matrix}\right)
=:\prod_{j=1}^{\stackrel{n}{\leftarrow}}\left(\begin{matrix}1&\alpha_j^*z^{-j}\\ \alpha_jz^{j} & 1 \end{matrix}\right) $$
This is $g_2$, except that we have dropped the numerical factors which enforce the algebraic condition $det(g_2)=1$ (which is an implicit regularization). Note that in Lemma \ref{szegorecursion}
\begin{equation}\label{xnotation} r_n=\gamma_2+\delta_2:=1+\sum_{k=1}^n x_kz^k \end{equation}
Whereas the zeroes of $p_n$ are in $\Delta$, the zeroes of $r_n$
are in $\Delta^*$, the complement of the closed unit disk. This implies
$f_+=-log(r_n)\in H^0(\Delta)$, because for a Verblunsky sequence with $\alpha_j=0$, $j>n$, $r_{n+i}=r_n$, $i\ge 0$, hence the limit in (\ref{verblunskymap3}) is expressed exactly in terms of $r_n$. To summarize, for Verblunsky sequences with a finite number of nonvanishing terms, the deterministic sense in which the diagram (\ref{vdiagram}) inverts $M$ is that there are bijective correspondences
\begin{equation}\label{vdiagramfin}
\begin{matrix} f_+=r_n &\leftrightarrow & \mu(\alpha)=\frac{\prod_{k=1}^n(1-|\alpha_k|^2)}{|r_n(e^{i\theta})|^2}\frac{d\theta}{2\pi}\\
\updownarrow & & \updownarrow\\
x=e^{-r_n} & \leftrightarrow & \alpha=(\alpha_1,...,\alpha_n,0,...) \end{matrix}
\end{equation}
By taking limits (for functions and measures having smooth boundary values), this implies

\begin{proposition}\label{classicalvdiagram} The diagram (\ref{vdiagram}) inverts $M$ when the $f_+$ and $x$'s are smooth up to the boundary, the $\alpha$ sequences are rapidly decreasing, and the measures $\mu$ have smooth positive densities on the circle.
\end{proposition}

The Szego identity (\ref{Szego}) implies that more generally the diagram is valid for $f_+$ and $x$ with $\sum_n n|f_n|^2<\infty$ and $\alpha$ such that $\sum_n n|\alpha_n|^2<\infty$ (There are also deterministic statements which apply to $l^2$ sequences, such as the `Szego condition', see \cite{Simon1} and \cite{Simon2}). However none of this
is applicable for
the random $f_+$ that occur in Theorem \ref{maintheorem}.

For the purposes of this paper, it is crucial that there are simple explicit expressions for the coefficients of $x$ in terms of $\alpha$. By convention, in the following statement, $\alpha_0=1$. The following is a straightforward calculation.

\begin{lemma}\label{volumelemma1}For rapidly decreasing sequences of Verblunsky coefficients, $x:=\gamma_2+\delta_2$, and $n=1,2,...$, $x_n$ is the sum of terms of the form
$$\alpha_{i(1)}\alpha_{j(1)}^*...\alpha_{i(L)}\alpha_{j(L)}^* $$
for some $L$, where the indices satisfy the constraints
$$i(1)>j(1)>i(2)>...>i(L)>j(L)\ge 0 \text{  and  }\sum_{u=1}^L(i(u)-j(u))=n$$
\end{lemma}

For example
$$x_1=\alpha_1+\alpha_2\alpha_1^*+\alpha_3\alpha_2^*+... $$
In general each $x_n$ is a multilinear function of the $\alpha_j$ and their conjugates. By contrast
each $x_n$ is a polynomial in the coefficients $f_j$ of $f_+\in H^0(\Delta,\mathbb C)/\mathbb C$.

The following definition is intended to capture basic intuition about the sum defining $x_n$ in terms of $\alpha$.

\begin{definition}\label{bulk/boundary} The bulk terms in the sum for $x_n$ in Lemma \ref{volumelemma1} are those for which the all the gaps $i(u)-j(u)=1$, i.e. those for which the length $L=n$. Terms for which $L<n$ are referred to as boundary terms. We refer to a pair $i(u)>j(u)=i(u)-1$ (with gap 1) as a 2-bit.  In general we refer to a pair $i(u)>j(u)=i(u)-m$ as an $m$-bit.
\end{definition}

Note that boundary terms arise when two or more 2-bits collide to form an $m$-bit, where $m>2$. For example the $2$-bits $\alpha_2\alpha_1^*$ and $\alpha_m\alpha_{m-1}^*$ collide when $m\downarrow 3$, to form the $3$-bit
$\alpha_m\alpha_1^*$, where we think of $\alpha_2$ and $\alpha_{m-1}^*$ as annihilating one another in the process.
This annihilation, or cancelation, process will play a central role in this paper.

\section{Definition of the Map M }\label{Mdefinition}

The main point of this section is to properly formulate the map $M$ in (\ref{vdiagram}).

Given a point $z\in\Delta$, there is an isometry
$$S^1 \stackrel{I_{z,\rho}}{\rightarrow}S(z,\rho)$$
where $S(z,\rho)$, the sphere centered at $z$ with radius $\rho$, is equipped with the induced non-Euclidean
metric normalized to have length $2\pi$ (the non-normalized length
is $\pi sinh(2\rho)$) (see (\ref{nonEmetric})).

Let $Comp(Prob(S^1))$ denote the set of all compact subsets
of $Prob(S^1)$; equipped with the Vietoris topology, $Comp(Prob(S^1))$ is a metrizable compact Hausdorff space.

Define $\widetilde M:H^0(\Delta)/\mathbb C \times \Delta \to Comp(Prob(S^1))$ by
$$\widetilde M(f_+,z)=\{\text{limit points of } \frac{1}{\mathfrak z(\rho)}|e^{f_+\circ I_{z,\rho}}|^2\frac{d\theta}{2\pi} \text{ as } \rho\uparrow\infty\}$$
and let
$$(H^0(\Delta)/\mathbb C)':=\{f_+:\widetilde M(f_+,z) \text{ is a point } \forall z\in \Delta\}$$

\begin{lemma}\label{workinglemma}Suppose $g\in PSU(1,1)$. Then

(a) $I_{g\cdot z,\rho}=g\circ I_{z,\rho} $

(b) $\widetilde M(g_*f_+,g\cdot z)=\widetilde M(f_+,z) $

(c) $\widetilde M(\cdot,z)\circ g=\widetilde M(g_*(\cdot),z)=\widetilde M(\cdot,g^{-1}\cdot z)$

(d) $(H^0(\Delta)/\mathbb C)'$ is stable with respect the action of $PSU(1,1)$.

(e) If $\mu\in Prob(H^0(\Delta)/\mathbb C)^{PSU(1,1)}$, then $\widetilde M(f_+,0)$ is $\mu$-almost surely a point iff for each $z\in\Delta$,
$\widetilde M(f_+,z)$ is $\mu$-almost surely a point.

(f) If $\mu$ as in (e) is ergodic, then the $\mu$ measure of $(H^0(\Delta)/\mathbb C)'$ is zero or one.

\end{lemma}

\begin{proof} (b) asserts that $\widetilde M$, as a function of $f_+$ and a basepoint, is $PSU(1,1)$ invariant.
The other parts are consequences of this.

\end{proof}

From now on we will fix the basepoint $z=0$, which will break $PSU(1,1)$-invariance.

\begin{definition}\label{Mdefinition} $M:H^0(\Delta)/\mathbb C \to Comp(Prob(S^1))$ is defined by $M(\cdot)=\widetilde M(\cdot,0)$, i.e.
$$M(f_+)=\{\text{limit points of } \frac{1}{\mathfrak z(r)}|e^{2Re(f_+(re^{i\theta}))}\frac{d\theta}{2\pi} \text{ as } r\uparrow 1\}$$
\end{definition}

\begin{remark} This is a nonlocal regularization. In physics it is important that the regularization of $e^f$ is local. One can fairly ask why we are deviating from the local multiplicative chaos regularization; the answer is that we have stumbled on this point of view in other contexts, and this leads to interesting questions.
\end{remark}

Since it is expressed as a limit, $M$ is a Borel map. However, $M$ is not continuous, and it is not $PSU(1,1)$-equivariant in any sense that we can identify (see Theorem \ref{noninvariance} below). Despite all these shortcomings, it is a very interesting map (we will consider some of its deterministic properties elsewhere).

It follows from (f) of Lemma \ref{workinglemma} that if $\nu$ is an ergodic $PSU(1,1)$-invariant distribution
on $H^0(\Delta)/\mathbb C$, then $M_*(\nu)(Prob(S^1))$ is zero or one. There is an enormous variety of such measures,
and for some $M_*(\nu)(Prob(S^1))$ is zero and for others it is one.

\begin{proposition}\label{techprop} For $\beta^2<2$, $M_*\nu_{\beta}(Prob(S^1))=1$.
\end{proposition}

The proposition is a corollary of difficult but well-known facts from the theory of multiplicative chaos (For recent expositions, see section 3.3 of \cite{AJKS2} or Theorem 1.3 of \cite{CN} (i.e. Theorem 1.2 of \cite{Berestycki}) and the subsequent commentary (Our notation is consistent with the first reference and at odds with the second). As in the statement of Theorem \ref{maintheorem}, our nonlocal regularization probably makes the restriction on $\beta$ unnecessary, but we have failed to resolve this.

There exists a large (possibly intractable) family of $PSU(1,1)$ invariant distributions on $H^0(\Delta,\mathbb C)/\mathbb C$. This should be contrasted with the following

\begin{theorem}\label{noninvariance} There does not exist a $PSU(1,1)$ invariant probability measure on $Prob(S^1)$.\end{theorem}

\begin{proof} Suppose otherwise. Then there exists an ergodic invariant probability measure. Relative to this ergodic measure, for a.e. $\mu\in Prob(S^1)$, $PSU(1,1)_{*}\mu$ is dense in $Prob(S^1)$. We claim that the weak$^*$ closure of $PSU(1,1)_{*}\mu$ is the union of $PSU(1,1)_{*}\mu$ and the set of $\delta$ measures around the circle. This will imply that the $PSU(1,1)$ orbit is not weak$^*$ dense, a contradiction.

There exists a point $z_0\in S^1$ with $\mu(\{z_0\})=0$. Given $0<r<1$,
$$\phi_1(z)=\frac{z+re^{-iq_1}}{1+re^{iq_1}z} $$
is a hyperbolic element of $PSU(1,1)$ with a repelling fixed point at $z_0$ and an attracting fixed point at $z_0^*$.
The limit of the measures $(\phi_1)_*\mu$ tends to $\delta_{z_0^*}$ as $\r\uparrow 1$. By applying a rotation if necessary, we can arrange for the limit to be any delta measure around the circle. Thus the delta measures are in the closure of the $PSU(1,1)$ orbit through $\mu$.

Conversely suppose that $\nu$ is in the closure of $PSU(1,1)_*\mu$. There exists a sequence $g_n \in PSU(1,1)$ such that $(g_n)_*\mu$ converges weakly to $\nu$. For each $n$ let $z_n\in \overline{\Delta}$
denote a fixed point for $g_n$. By passing to a subsequence if necessary, we can suppose that $z_n \to z_0 \in \overline{\Delta}$.

If $z_0\in \Delta$, then we can suppose $z_0=0$. A subsequence of the $g_n$ will then converge to a rotation. In this case $\nu$ is just a rotation of $\mu$ and we are done.

So suppose that $z_0\in S^1$. In this event it is convenient to switch to the upper half space with $z_0=\infty$.

By passing to a subsequence if necessary, we can suppose all of the $g_n$ are parabolic, or all of the $g_n$ are hyperbolic. In the parabolic case the $g_n$ are essentially horizontal translations. If there is a finite limit, then $g_n$ converges to a parabolic element. Otherwise $\nu$ is the delta measure at $\infty$.

Suppose the $g_n$ are hyperbolic. Each $g_n$ has a second fixed point $z_n'$. By passing to a subsequence we can suppose this second sequence converges. Suppose these second fixed points converge to $z_0$. Then we are essentially back in the parabolic case.

Suppose that the $z_n'$ converge to a second point $z_1$ on the circle. We can suppose this second point is $z_1=0$ (in the upper half plane model). In this event the $g_n$ are dilations. If there is a finite limit for the magnitude of dilations, then the $g_n$ converge
to a hyperbolic element and $\nu$ is in the $PSU(1,1)$ orbit of $\nu$. If for some subsequence, the magnitude goes to zero or one, then $\nu$
is a delta measure at either $0$ or $\infty$.

This completes the proof.

\end{proof}

\begin{corollary} There does not exist a $PSU(1,1)$ invariant probability measure on $S^1\backslash Homeo(S^1)$, where $PSU(1,1)$ acts by composition on the right (where $S^1$ is identified with rotations and homeomorphisms are assumed to be orientation preserving).

\end{corollary}

\begin{proof} There is a $PSU(1,1)$ equivariant isomorphism
$$\{\mu\in Prob(S^1):\mu_d=0,support(\mu)=S^1\}\leftrightarrow S^1\backslash Homeo(S^1)$$
(given a homeomorphism, the generalized derivative divided by $2\pi$ is a probability measure,
and given an atomless probability measure with full support, the corresponding cumulative distribution function (unique up to a shift)
is a homeomorphism).
Thus the corollary follows from the theorem.
\end{proof}

\section{Volume Using Root Subgroup Factorization}\label{RSFvolume}

In this section we use the notation established in Section \ref{RSF}. Throughout we fix $N<\infty$,
and by convention, $x_0=1$.

\begin{lemma} The map
$$\prod_{n=1}^N\Delta \to \prod_{n=1}^N\mathbb C:(\alpha_1,...\alpha_N) \to (x_1,...,x_N) $$
where $1+\sum_{n=1}^Nx_nz^n=\gamma_2+\delta_2$, is an injective map onto an open domain properly
contained in
$$\{(x_1,...,x_N):x=1+\sum_{n=1}^N x_nz^n \text{ is nonvanishing in }closure(\Delta) \}$$

\end{lemma}

\begin{proof}Abbreviate $\alpha=(\alpha_1,..,\alpha_N,0,...)$.  First note that $x$ is nonvanishing. This is because $x=\gamma_2+\delta_2$ is the $N$th reversed polynomial for the measure corresponding to $\alpha$.

We now show that the map is injective. Given $\alpha$ mapping to $x_1,..,x_N$, we obtain a corresponding $f_+=-log(x)$ and a measure
$\mu=M(f_+)$ corresponding to $\alpha$. So the map is injective. The fact that the map  $\alpha \to (x_1,...,x_N)$ is an injective polynomial map implies that the derivative is injective. This implies that the image is open.
Since the $\alpha_j$ are bounded by one, the image is bounded, hence it is properly contained in the set of
$x=1+\sum_{n=1}^Nx_nz^n$ which are nonvanishing in the closure of $\Delta$ (Unfortunately it is not clear how to precisely describe the image).

\end{proof}

\begin{theorem}
Consider the composition of maps
$$(\alpha_1,..,\alpha_N) \to x=\gamma_2+\delta_2 \to f_+ \to (f_1,...,f_N)$$.

(a) For the bijective polynomial map $f_+=\sum_{n=1}^N f_nz^n \to (x_1,...,x_N)$, where $x=exp(-f_+)$,
$$\left(\prod_{n=1}^N d\lambda(f_n)\right)=\prod_{n=1}^N d\lambda(x_n) $$

(b)
$$\chi_{\tilde D_N}(x_1,...,x_N)\prod_{n=1}^N d\lambda(x_n)=\prod_{n=1}^N(1-|\alpha_n|^2)^{n-1}d\lambda(\alpha_n) $$
where $\tilde D_N$ is the image of the map $(\alpha_1,...,\alpha_N)\to (x_1,...,x_N)$.

(c) $$\chi_{D_N}(f_1,...,f_N)\prod_{n=1}^N d\lambda(f_n)= \prod_{n=1}^N(1-|\alpha_n|^2)^{n-1} d\lambda(\alpha_n) $$
where $D_N$ is the image of the map $(\alpha_1,...,\alpha_N)\to (f_1,...,f_N)$.

\end{theorem}

\begin{proof} (a)  There is a simple triangular relationship between the $x_n$ and the $f_n$,
namely $f_n=x_n+p(x_1,..,x_{n-1})$, where $p$ is a polynomial (e.g. $f_1=x_1$, $f_2=x_2-\frac12 x_1^2$, ...). Consequently (in terms of differential forms)
$$df_1\wedge ...\wedge df_N=dx_1\wedge ...\wedge dx_N$$
and similarly for volume.

(b) It is straightforward to check this for small $N$. For $N=1$, $x_1=\alpha_1$. For $N=2$,
$x_1=\alpha_1+\alpha_2\overline{\alpha_1}$, $x_2=\alpha_2$, and the result is obvious.

Suppose the result holds for $N-1$. Write $y_j$ for $x_j(\alpha_1,\overline{\alpha_1},..,\overline{\alpha_{N-1}},0...)$ and $x_j$ for $x_j(\alpha_1,\overline{\alpha_1},...,\overline{\alpha_N},0,...)$. Thus we are assuming (in terms of differential forms)
\begin{equation}\label{inductionstep}dy_1\wedge d\overline{y_1}\wedge ...\wedge dy_{N-1}\wedge d\overline{y_{N-1}}=\prod_{n<N}(1-|\alpha_n|^2)^{n-1}
d\alpha_1\wedge d\overline{\alpha_1}\wedge...\wedge d\alpha_{N-1}\wedge d\overline{\alpha_{N-1}} \end{equation}

Lemma \ref{volumelemma1} implies
$$x_j=y_j+\alpha_N\overline{y_{N-j}} $$
In particular $x_N=\alpha_N$. Therefore
$$dx_1\wedge d \overline{x_1}\wedge ...\wedge dx_N \wedge d \overline{x_N}=$$
$$(dy_1+\alpha_Nd\overline{y_{N-1}}) \wedge (d\overline{y_1}+\overline{\alpha_N}dy_{N-1})\wedge...\wedge (dy_{N-1}+\alpha_Nd\overline{y_{1}}) \wedge (d\overline{y_{N-1}}+\overline{\alpha_N}dy_{1}) \wedge d\alpha_N\wedge d\overline{\alpha_N}   $$
Now consider terms obtained by expanding this product of factors as a sum (similar to what one does in proving the binomial formula).
One term is obtained by choosing $dy_j$ (or its conjugate) from each of the factors; we can evaluate this using (\ref{inductionstep}) and wedging this with $d\alpha_N\wedge d\overline{\alpha_N}$. As an example of how we could obtain another nonzero term, we could
choose $dy_1$ from the first factor and $\alpha_{N} d\overline y_{1}$ from the penultimate factor; then we would be forced to choose
$\overline{\alpha_{N}} d y_{N-1}$ from the second factor, and $d \overline{y_{N-1}}$ from the last factor; and then we must make further choices involving other variables. The upshot is that each of the other nonzero terms is of the following form, up to a sign:
$$\wedge_{r=1}^{\stackrel{k}{\rightarrow}}\left(dy_{i_r}\wedge \alpha_{N} d\overline y_{i_r}\wedge dy_{N-i_r}\wedge \overline{\alpha_{N}} d\overline y_{N-i_r} \right)\wedge_{s=1}^{\stackrel{N-k}{\rightarrow}}\left(dy_{j_s}\wedge d\overline{y_{j_s}}\right)\wedge d\alpha_N\wedge d\overline{\alpha_N}$$
where $i_1<..<i_k<N$, $j_1<...<j_{N}<N$, and the $\{i_r\}$ and $\{j_s\}$ are disjoint. Using the induction step, up to sign, this equals
$|\alpha_N|^{2k} $ times the right hand side of (\ref{inductionstep}).
The number of ways in which we could obtain a term of this form is $\left(\begin{matrix} N\\k\end{matrix}\right)$. When we account for signs, we can use the binomial formula to add these terms over $k=0...N-1$ to obtain (b).

(c) follows from (a) and (b). This proves the theorem.

\end{proof}

\section{The Gaussian Distribution for $x$}\label{proof1}

The strategy of the proof of Theorem \ref{maintheorem} is to show that the two distributions in the statement of theorem induce the same distribution for the coefficients $x_1,x_2,...$, where if we start with the $\nu_{\beta}$ distribution for $f_+$, then $x=exp(-f_+)$, and if we start with the $\alpha$ distribution, then $x=\gamma_2+\delta_2$ (as in Section \ref{RSF}). In this section we do the easier Gaussian calculation, and we will see that
the answer has an interesting interpretation.
In the next section we will do the other (Fermionic) calculation.

Throughout this section we will use multi-index notation. For a multi-index $p=(p(1),p(2),...)$, length $L(p):=|\{j:p(j)>0\}|$, $p!:=\prod_j p(j)!$, $l^1$ norm $|p|:=\sum_j p(j)$, and $deg(p):=\sum_j jp(j)$.

\subsection{The Gaussian Case}

In this subsection we suppose that $f_1,...$ (the coefficients of $f_+$) are distributed according to
$$d\nu_{\beta}=\prod_{n=1}^{\infty}\frac{n\beta}{\pi}e^{-n\beta|f_n|^2}d\lambda(f_n) $$
Note that
$$E(|f_n|^{2k})= \frac{k!}{(n\beta)^k} \text{ and } E(\prod_{n\ge 1} f_n^{p(n)}(\prod_{m\ge 1} f_m^{q(m)})^*)= \prod_{n\ge 1} \frac{p(n)!}{(n\beta)^{p(n)}}$$
if $p=q$ and zero otherwise. In multi-index notation, with $f=(f_1,f_2,...)$,
\begin{equation}\label{gaussmoments}E(f^p(f^q)^*)= \frac{p!}{\prod_n n^{p(n)}}\beta^{-|p|} \end{equation}
if $p=q$ and zero otherwise.

\begin{theorem} \label{gausscase1} Suppose that $x=exp(-f_+)$. If $deg(p)=deg(q)$, then
$$E(x^p(x^q)^*)=\sum \frac{1}{\prod_{n,r} (J_{n,r}!)\prod_{m,s}(K_{m,s}!)}
\frac{(\sum_{n,r} J_{n,r})!}{\prod_u u^{(\sum_{n,r} J_{n,r})(u)}}\beta^{-|\sum_{n,r} J_{n,r}|} $$
where the sum is over all multi-indices $J_{n,r},K_{m,s}$ which satisfy $deg(J_{n,r})=n$,
$deg(K_{m,s})=m$, $\sum_{n,r} J_{n,r}=\sum_{m,s} K_{m,s}$ (equality of two multi-indices),
and for given indices $n,m>0$, $1\le r\le p(n)$, $1\le s\le q(m)$.

If $deg(p)\ne deg(q)$, then the expectation vanishes.

\end{theorem}

\begin{examples}\label{gaussexamples} Here are two extreme examples:

(a) Because $x_1=-f_1$ is Gaussian, $E(x_1^p(x_1^p)^*)=p!\beta^{-p}$, $p=0,1,...$.

(b) In the next subsection we will see that for a general multi-index $p$ with $d=deg(p)$, $E(x^px_d^*)$ is the probability generating function for a sum of
independent Bernoulli random variables.

\end{examples}

\begin{proof} The group of rotations of the circle acts naturally on the various objects in the statement of the theorem:
rotation by $\theta$ acts on $f_+$ ($x$) by $(f_n)\to (e^{-in\theta}f_n)$ ($(x_n) \to (e^{-in\theta}x_n)$, respectively),
the map $f_+ \to x=exp(-f_+)$ is equivariant, and the measure $\nu_{\beta}$, hence its expectation, is invariant.
Consequently if $deg(p)\ne deg(q)$, then rotations act on $x^p(x^q)^*$ by a nontrivial character, and the expectation in the theorem must vanish. We henceforth assume $deg(p)=deg(q)$.

$f_+$ and $-f_+$ have the same distribution, so in what follows we will calculate the distribution for
the coefficients of $y=exp(f_+)$; these distributions will be the same as for the coefficients of $x=exp(-f_+)$.
This notational change will eliminate signs.

Recall
$$exp(\sum_{n>0}f_nz^n)=\sum_{n\ge 0}y_nz^n \text{  where  }
y_n=\sum_{\{J:deg(J)=n\}}\frac{1}{J!}f^J $$
where $J=(J(1),J(2),...)$ is a multi-index. Hence
\begin{equation}\label{sum1}E(y^p(y^q)^*)
=E(\prod_{n\ge 0} (\sum_{J_n}\frac{1}{J_n!}f^{J_n})^{p(n)}(\prod_{m\ge 0} (\sum_{K_m}\frac{1}{K_m!}f^{K_m})^{q(m)})^*) \end{equation}
where $J_1,J_2,...,K_1,...$ are multi-indices satisfying the constraints $deg(J_n)=n$, $deg(K_m)=m$. In order to take the $p(n)$ power, it is convenient to introduce independent copies of the multi-index $J_n$, which we denote by $J_{n,1},..,J_{n,p(n)}$ (each has degree $n$), as in the statement of the theorem. Then (\ref{sum1})
\begin{equation}\label{sum2}=\sum_{J_{n,r},K_{m,s}}\frac{1}{(\prod_{n,r} J_{n,r}!)(\prod_{m,s} K_{m,s}!)}E\left(f^{\sum_{n,r} J_{n,r}}(f^{\sum_{m,s} K_{m,s}})^*\right) \end{equation}
where the multi-indices satisfy the constraints $deg(J_{n,r})=n$, $deg(K_{m,s})=m$. (\ref{gaussmoments}) implies this \begin{equation}\label{sum3}=\sum \frac{1}{\prod_{n,r} (J_{n,r}!)\prod_{m,s}(K_{m,s}!)}
\frac{(\sum_{n,r} J_{n,r})!}{\prod_u u^{(\sum_{n,r} J_{n,r})(u)}}\beta^{-|\sum_{n,r} J_{n,r}|} \end{equation}
where the sum is over all multi-indices $J_{n,r},K_{m,s}$ which satisfy the constraints $deg(J_{n,r})=n$, $deg(K_{m,s})=m$, $\sum_{n,r} J_{n,r}=\sum_{m,s} K_{m,s}$ (equality of two multi-indices), and for given indices $n,m>0$, $1\le r\le p(n)$, $1\le s\le q(m)$. This
completes the proof.

\end{proof}

\subsection{Combinatorial Interpretations}

Note that
$$\sum_{n,r,u}J_{n,r}(u)\le \sum_{n,r,u}uJ_{n,r}(u)=\sum_n np(n)=deg(p)$$
Therefore the sum in Theorem \ref{gausscase1} is a polynomial in $\beta^{-1}$ of the form
\begin{equation}\label{laurentpoly}\sum_{k=1}^{deg(p)} a(p,q,k)\beta^{-k} \end{equation}
where
\begin{equation}\label{coeffdfn}a(p,q,k)=\sum \frac{1}{\prod_{n,r} (J_{n,r}!)\prod_{m,s}(K_{m,s}!)}
\frac{(\sum_{n,r} J_{n,r})!}{\prod_u u^{(\sum_{n,r} J_{n,r})(u)}}\end{equation}
and the sum is over all multi-indices $J_{n,r},K_{m,s}$ which satisfy $deg(J_{n,r})=n$,
$deg(K_{m,s})=m$, $\sum_{n,r} J_{n,r}=\sum_{m,s} K_{m,s}$ (equality of two multi-indices),
for given indices $n,m>0$, $1\le r\le p(n)$, $1\le s\le q(m)$, and $\sum_{n,r,u}J_{n,r}(u)=k$.
We will often suppress $p,q$ and write $a_k=a(p,q,k)$.

A trivial observation is that the $a(p,q,k)$ are nonnegative. Consequently (\ref{laurentpoly}) is the mass generating function for a finite positive measure, and in some cases it is a probability mass generating function (in the variable $\beta^{-1}$).

\subsubsection{A Special Case: Variance}

We first consider a pivotal special case, the variance $E(x_nx_n^*)$. In this case
Theorem \ref{gausscase1} specializes to
\begin{equation}\label{laurentpoly2}E(x_nx_n^*)=\sum_{k=1}^{n} a_k^{(n)}\beta^{-k} \end{equation}
where
\begin{equation}\label{coeffdfn2}a_k^{(n)}=\sum \frac{1}{J!\prod_u u^{J(u)}}\end{equation}
and the sum is over multi-indices $J$ which satisfy $deg(J)=n$ and $|J|=k$ (In the notation of (\ref{coeffdfn}),
$p(n)=1=q(n)$, $p$ and $q$ vanish otherwise, and $J:=J_{n,1}=K_{n,1}$).
We will see that in terms of $\beta^{-1}$, this is a probability mass function, and it has many different combinatorial interpretations.

Let $S_n$ denote the symmetric group on $n$ letters $\{1,...,n\}$. There are natural inclusions $S_n \subset S_{n+1}$, where $S_n$ is identified with the permutations which fix $n+1$. Let $\mathcal P_n$ denote the set of partitions of $n$. On the one hand
$\mathcal P_n$ can be identified with the conjugacy classes of $S_n$, because the conjugacy class of an element is determined by the cardinalities of its orbits. On the other hand a partition of $n$ can be identified with its density, which is a multi-index of degree $n$. Consider the composition of maps
\begin{equation}\label{randomvariable2}S_n \to S_n/conj\leftrightarrow \mathcal P_n \leftrightarrow \{J:deg(J)=n\} \stackrel{|\cdot|}{\rightarrow} \{1,2,...,n\}\end{equation}
where the first map sends a group element to its conjugacy class, the second and third maps are equivalences (as described above), and the fourth map sends a density $J$ to $|J|=\sum_u J(u)$; the composition sends a permutation to the number of its orbits.
The normalized Haar measure for $S_n$ pushes forward to the probability measure which attaches weight $|C|/n!$ to a conjugacy class $C$.
This is most easily computed by realizing that it is the same as the reciprocal of the size of the stability subgroup of a representative for the conjugacy class; the cardinality of the stabilizer is $\prod_u J(u)!u^{J(u)}$, where $J$ is the corresponding density (The factorial corresponds to permuting the $J(u)$ orbits of size $u$, and the other factor corresponds to the fact that a permutation stabilizing an orbit can map a given element to any of the $u$ elements in the orbit).

\begin{lemma} (\ref{laurentpoly2}) is the probability mass function (in the variable $\beta^{-1}$) for the random variable
(\ref{randomvariable2}), i.e.
$$\sum \prod_{u=1}^n\frac{1}{J(u)!u^{J(u)}}=1$$
where the sum is over all multi-indices $J$ such that $deg(J)=n$. More precisely
$$\sum \prod_{u=1}^n\frac{1}{J(u)!u^{J(u)}}=\frac 12$$
where the sum is over all multi-indices $J$ such that $deg(J)=n$ and $|J|=\sum J(u)$ is odd.

\end{lemma}

\begin{proof} The first statement follows from the preceding discussion. The second statement uses
the existence of a normal degree two subgroup $A_n\subset S_n$.
\end{proof}

\begin{theorem}\label{variancepmf}
$$E(x_nx_n^*)=\prod_{k=1}^n (\frac{1}{k}\beta^{-1}+\frac{k-1}{k})=\left(\begin{matrix}\beta^{-1}+n-1\\n\end{matrix}\right)$$
i.e. this is the pmf for a sum of $n$ independent Bernoulli random variables with harmonic parameters $p_k=\frac{1}{k}$, $k=1,...,n$.
\end{theorem}

The following is elementary (and probably well-known).

\begin{proposition}\label{recursionchar}
If $X_1,...$ is a sequence of integer valued random variables, and $G_n(s)$ denotes the probability mass function for the partial sum $\sum_{k=1}^nX_k$, then the following two statements are equivalent: (1)
$$G_n(s)=\prod_{k=1}^n(\frac 1k s+\frac{k-1}{k})$$
i.e. the $X_k$ are independent Bernoulli random variables with harmonic parameters
$p_k=\frac 1k$, and (2) the $G_n$ satisfy the recursion
$$G_n(s)=\frac{1}{n}\left(\sum_{k=0}^{n-1}G_{k}(s)\right) s$$
(where $G_0=1$).
\end{proposition}

Thus the theorem can also be stated in the following equivalent form.

\begin{theorem}\label{variancepmf2}
$$E(x_nx_n^*)=\frac{1}{n}\sum_{k=0}^{n-1}E(x_kx_k^*)\beta^{-1}$$
(recall that $x_0=1$)\end{theorem}

The main point, spelled out in the proof, is that we can realize the Bernoulli variables in an interesting way.

\begin{proof} (of Theorem \ref{variancepmf}) Recall that the space of virtual permutations is
the inverse limit of the system of projections
$$...\to S_{n+1} \stackrel{p_{n+1}}{\rightarrow} S_n \stackrel{p_n}{\rightarrow} ...\to S_1$$
where $p_n(\sigma\in S_n)=\sigma$ if $\sigma\in S_{n-1}\subset S_n$ and $p_n(\sigma)(k)=\sigma(k)$ if $\sigma(k)<n$
and $p_n(\sigma)(k)=\sigma(n)$ if $\sigma(k)=n$. In terms of the cycle structure of $\sigma$, $p_n$ simply has the effect of deleting $n$ from the relevant cycle. The important point is that with respect to these projections, the Haar measures are coherent; see section 1 of \cite{KOV}.

Let $X_k$ denote the Bernoulli random variable with parameter $p=\frac 1k$, originally defined on $S_k$ but extended to the space of
virtual permutations using the above projections, given
by $X_k(\sigma)=1$ if $\sigma(k)=k$ and $X_k(\sigma)=0$ otherwise. Note
$$G_{X_k}(\beta^{-1})=\frac 1k \beta^{-1}+\frac{k-1}{k}$$

We claim that on $S_{n}$, (1) $|J|=\sum{k=1}^n X_k$ and (2) $X_1,...,X_n$ are independent. We prove the first claim by induction on $n$. So we suppose that $X_1+...+X_{n-1}=|J|$ on $S_{n-1}$ and we have to show that $(X_1+...+X_{n-1})\circ p_{n}+X_{n}=|J|$ on $S_{n}$. If $\sigma(n)= n$,
then the number of cycles for $\sigma$ as an element of $S_{n}$ is one greater than viewed as an element of $S_{n-1}$. Since $X_{n}(\sigma)=1$ in this case, we have equality. If $\sigma(n)\ne n$, then
$X_{n}(\sigma)=0$ and the number of cycles stays the same. Thus we again have equality. This prove (1).

Now consider (2). We assume (2) holds for $n-1$. Because we have established (1), $G_{|J|}$ factors as a product of $G_{X_k}$, $k=1,...,n-1$, on $S_{n-1}$. It therefore suffices to show that $|J|\circ p_n$ and $X_n$ are independent, i.e.
\begin{equation}\label{claim2}P(\{|J|\circ p_n=r\}\cap \{X_{n}=\delta\})=P(\{|J|\circ p_n=r\})P(\{X_{n}=\delta\})\end{equation}
where $\delta=1$ or $=0$.
(for this will imply that $G_{|J|}$ factors on $S_n$). This is a counting problem, because the probabilities are computed using the counting probability measure on $S_n$. Suppose that $\delta=1$. In this case the right hand side
 equals $=\frac 1n P(\{|J|\circ p_n=r\})$. On the left hand side we are counting the number of $\sigma\in S_{n-1}\subset S_n$ having $r$ cycles. For the probability on the right we have to count the number
of $\eta\in S_n$ such that $p_n(\eta)$ has $r$ cycles. If $\eta(n)\ne n$, then $\eta$ is obtained from an $\eta'\in S_{n-1}$ having $r$ cycles by inserting $n$ into one of the cycles. This can be done in $n-1$ ways. This shows
that
$$n|\{|J|\circ p_n=r\}\cap \{X_{n}=\delta\}|=|\{|J|\circ p_n=r\}|$$
This proves (2) for $\delta=1$. The case $\delta=0$ follows automatically.

Given (1) and (2), it follows that on $S_n$
$$G_{|J|}(\beta^{-1})=\prod_{k=1}^n(\frac 1k \beta^{-1}+\frac{k-1}{k})=\frac {1}{n!}\prod_{k=1}^n(\beta^{-1}+k-1)
=\left(\begin{matrix}\beta^{-1}+n-1\\n\end{matrix}\right)$$
This proves Theorem \ref{variancepmf}.
\end{proof}

Yet another way to state the theorem is the following:

\begin{corollary}
$$a_k^{(n)}=\frac{1}{n!}\sum l_1...l_{n-k}$$
where the sum is over all choices of $n-k$ numbers $0\le l_1<l_2<...<l_{n-k}< n$,
and the vacuous sum is $=1$.
\end{corollary}

\subsubsection{The General Case I}

In this subsection we consider the constraints on the terms in the the formula for $E(x^p(x^q)^*)$ in Theorem \ref{gausscase1}, when $d=deg(p)=deg(q)$. These can be visualized in terms of the following diagram:

$$\begin{matrix}\prod_{\{p(n)>0\}} \left(\mathcal P_n \times ...\times \mathcal P_n\right)&\times &\prod_{\{q(m)> 0\}}\left(\mathcal P_m\times...\times \mathcal
P_m\right)\\
\downarrow & & \downarrow\\
\mathcal P_{d} & \times & \mathcal P_{d}\end{matrix}
\qquad
\begin{matrix} ((J_{n,r}), (K_{m,s}))\\
\downarrow \\
(\sum_{n,r}J_{n,r},\sum_{m,s}K_{m,s}) \end{matrix}  $$
where we are identifying partitions with their densities (which can be added), and in the diagram for each $n$ ($m$)
there are $p(n)$ copies of $\mathcal P_n$ ($q(m)$ copies of $\mathcal P_m$, respectively).

At the base we must impose the constraint $\sum_{n,r}J_{n,r}=\sum_{m,s}K_{m,s} $.
This means that we are really interested in the diagonal
$$\Delta(\mathcal P_{d})\subset \mathcal P_{d} \times \mathcal P_{d}$$
and its inverse image. Note that the map $(K_{m,s})\to K=\sum_{m,s}K_{m,s}$ is generally far from surjective,
because $K_{m,s}(v)=0$ for $v>m$ (since $deg(K_{m,s})=m$). This implies that $K(v)=0$ for $v>\max supp(q)$.
This implies the following

\begin{lemma}\label{Lconstraint}The image of the projection intersected with the diagonal consists of pairs $(L,L)=(J,K)\in \Delta(\mathcal P_{d})$ such that $L(u)=0$ for $u>u_0$ where $u_0=\min\{\max supp(p),\max supp(q)\}$. \end{lemma}

\subsubsection{Multiplicity Free Case}

There is a simple probabilistic interpretation in the following generalization of the variance case (In the next section we will see that this corresponds to a multiplicity free condition from the $\alpha$ distribution point of view).

\begin{theorem}\label{generalizationformula} If $d=deg(p)$, then the expected value $E(x^p (x_d)^*)$ is the generating function for a sum of independent Bernoulli random variables. More precisely
$$E(x^p (x_d)^*)=\prod_{n\ge 1} \left(\prod_{k=1}^{n}(\frac{1}{k}\beta^{-1}+\frac{k-1}{k})\right)^{p(n)} $$

\end{theorem}

\begin{proof}In this case
$$\sum_{n,r}J_{n,r}=K_{d,1} $$
i.e. in the diagram in the previous subsection we can determine $K_{d,1}$ from the image in the diagonal $\Delta(\mathcal P_d)$. Furthermore there are not any constraints on
the $J_{n,r}$ beyond the conditions $deg(J_{n,r})=n$ and $r\le p(n)$. Therefore
$$E(x^p (x_d)^*)=\sum \frac{K_{d,1}!}{\prod_{n,r}J_{n,r}!K_{d,1}!\prod_{n,r,u} u^{J_{n,r}(u)}}\beta^{-|K_{d,1}|} $$
$$=\sum \frac{1}{\prod_{n,r}J_{n,r}!\prod_{n,r,u} u^{J_{n,r}(u)}}\beta^{-\sum_{n,r,u}J_{n,r}(u)}$$
$$=\prod_{n,r} \sum_{u} \frac{1}{\prod J_{n,r}(u)!u^{J_{n,r}(u)}}\beta^{-J_{n,r}(u)}$$
$$=\prod_{n\ge 1} \left(\prod_{k=1}^{n}(\frac{1}{k}\beta^{-1}+\frac{k-1}{k})\right)^{p(n)} $$
where the last step uses Theorem \ref{variancepmf}.

\end{proof}

One might suspect that if $E(x^p(x^q)^*)$ is the generating function of a probability measure, then either $|p|=1$ or $|q|=1$. This is false. For example
$$E(x_1x_2(x_1x_2)^*)=\frac 34 \beta^{-3}+\frac 14 \beta^{-2} $$
is a probability generating function.

\subsubsection{The General Case II}

We now want to attach an interpretation to the formula for $E(x^p(x^q)^*)$ in Theorem \ref{gausscase1}. In general $E(x^p(x^q)^*)$ is the generating function for a finite measure,
but $E(x^p(x^q)^*)\vert_{\beta=1}=\sum a_k$ is not necessarily one (e.g. $E(x_1^p(x_1^p)^*)=p!\beta^{-p}$), and $E(x^p(x^q)^*)$ does not have a simple factorization as in Theorem \ref{generalizationformula} (e.g.
$$E(x_2^2(x_2^2)^*)=\frac 32 \beta^{-4}+\beta^{-3}+\frac 12 \beta^{-2} $$
has two complex roots).

\begin{notation} Define $f(p,\cdot):\mathcal P_d\to \{0,1,...\}$ by
$$f(p,L)= \sum_{J_{n,r}}\left(\begin{matrix}L\\(J_{n,r})\end{matrix}\right)$$
when $L$ can be written as $L=\sum_{n,r}J_{n,r}$ with $deg(J_{n,r})=n$, $r\le p(n)$,
and $f(p,L)=0$ otherwise, where
$$\left(\begin{matrix}L\\(J_{n,r})\end{matrix}\right)=
\prod_{u}\left(\begin{matrix}L(u)\\(J_{n,r}(u))\end{matrix}\right)
=\prod_{u}\frac{L(u)!}{\prod_{n,r}J_{n,r}(u)!} $$

Also let $E_{\mathcal P_d}$ denote expectation with respect to the probability measure on $\mathcal P_d$
induced by normalized Haar measure on the symmetric group $S_d$.
\end{notation}

\begin{theorem}\label{bosonicsum}Assume $d=deg(p)=deg(q)$. Then
$$E(x^p(x^q)^*)=\sum_{k=1}^d a_k\beta^{-k}=E_{\mathcal P_d}\left( f(p,L)f(q,L)\beta^{-|L|}\right)$$
\end{theorem}

Note that in the variance case $p=\delta_d=q$, $f(p,\cdot)=f(q,\cdot)=1$.

\begin{proof}
$$E(x^p(x^q)^*)=\sum \frac{L!}{\prod_{n,r}J_{n,r}!\prod_{m,s}K_{m,s}!\prod_{u} u^{L(u)}}\beta^{-|L|} $$
$$=\sum_{L}\frac{L!}{\prod_u u^{L(u)}}\left(\sum \frac{1}{\prod_{n,r}J_{n,r}!\prod_{m,s}K_{m,s}!}\right)\beta^{-|L|}$$
where $L$ is constrained as in Lemma \ref{Lconstraint} and the inner sum is over $((J_{n,r}),(K_{m,s}))$ such that
$L=\sum_{n,r}J_{n,r}=\sum_{m,s}K_{m,s}$. This inner sum can be factored:
$$=\sum_{L}\frac{1}{L!\prod_u u^{L(u)}}\left(\sum_{J_{n,r}} \left(\begin{matrix}L\\(J_{n,r})\end{matrix}\right)\right)\left(\sum_{K_{m,s}} \left(\begin{matrix}L\\(K_{m,s})\end{matrix}\right)\right)\beta^{-|L|}$$
This is equivalent to the statement of the theorem.

\end{proof}

\section{The $\alpha$ Distribution for $x$}\label{proof2}

In the remainder of the paper we suppose that the random variables $\alpha_1,...\in\Delta$ are distributed
according to
\begin{equation}\label{alphadistribution}\prod_{n=1}^{\infty}\frac{n\beta}{\pi}(1-|\alpha_n|^2)^{n\beta-1}d\lambda(\alpha_n) \end{equation}
and (for notational convenience) $\alpha_0=1$.
We compute
$$E(|\alpha_n|^{2K})=2\pi\int_{r=0}^1r^{2K}\frac{n\beta}{\pi}(1-r^2)^{n\beta-1}rdr$$
$$=n\beta B(K+1,n\beta)=\frac{\Gamma(K+1)\Gamma(n\beta+1)}{\Gamma(n\beta+K+1)}
=\frac{K!}{(n\beta+1)...(n\beta+K)}$$
where $K$ is not necessarily integral in the first two expressions, and $B$ denotes the beta function.
In particular
\begin{equation}\label{L1norm} E(|\alpha_n|)=\frac{\Gamma(\frac 32)\Gamma(n\beta+1)}{\Gamma(n\beta+\frac 32)}\sim \frac 1n \text{ as }n\to\infty
 \end{equation}
The moments of $\alpha$ are given by
\begin{equation}\label{alphamoments}E( \alpha^p(\alpha^q)^*)=\prod_{n\ge 1}\frac{\Gamma(p(n)+1)\Gamma(n\beta+1)}{\Gamma(n\beta+p(n)+1)}
=\prod_{n\ge 1}\frac{p(n)!}{(n\beta+1)...(n\beta+p(n))}
\end{equation}
if $p=q$ and zero otherwise.

Recall from Lemma \ref{volumelemma1} that the $n$th coefficient of $x=\gamma_2+\delta_2$ is expressible as an infinite series consisting of terms of the form
\begin{equation}\label{xformulas}\alpha_{i(1)}\alpha_{j(1)}^*...\alpha_{i(L)}\alpha_{j(L)}^*\end{equation}
where $L=length(i)$, the indices are decreasing, $i(1)>j(1)>...>j(L)\ge 0$, and the sum of the gaps $\sum_u (i(u)-j(u))=n$.

\subsection{Calculation of $E(x^p(x^q)^*)$}

If $deg(p)\ne deg(q)$, then the expectation $E(x^p(x^q)^*)$ vanishes by rotational symmetry of the expectation and the fact that rotations act on $x^p(x^q)^*$ by a nontrivial character.

Suppose that $deg(p)=deg(q)$ is fixed. Then
\begin{equation}\label{sum11}E(x^p(x^q)^*)
=E(\prod_{n} (\sum \alpha_{i_n(1)}\alpha_{j_{n}(1)}^*...)^{p(n)}\prod_{m} (\sum
(\alpha_{k_m(1)}\alpha_{l_{m}(1)}^*...)^{q(m)})^*) \end{equation}
where the $i_n,j_n$ (and also $k_m,l_m$) are multi-indices satisfying the constraints following (\ref{xformulas}). In order to take the $p(n)$ power, it is convenient to introduce independent copies of these multi-indices, which we denote by $i_{n,1},..$ (and similarly for $j,k,l$). Then (\ref{sum11})
\begin{equation}\label{sum12}=\sum_{i_{n,r},j_{n,r},k_{m,s},l_{m,s}} E\left(\prod_{n,r} (\alpha_{i_{n,r}(1)}\alpha_{j_{n,r}(1)}^*...)(\prod_{m,s} (
\alpha_{k_{m,s}(1)}\alpha_{l_{m,s}(1)}^*... ) )^*\right) \end{equation}
where the multi-indices satisfy the constraints as in (\ref{xformulas}).

Now consider one of the terms in this sum.  (\ref{alphamoments}) implies that for this expected value to be nonzero,
this term must have the form
\begin{equation}\label{formofterm}\left(\prod_{n,r} (\alpha_{i_{n,r}(1)}\alpha_{j_{n,r}(1)}^*...)\right)\left(\prod_{m,s} (
\alpha_{k_{m,s}(1)}\alpha_{l_{m,s}(1)}^*...  )\right)^*=\prod_{N\ge 1}|\alpha_N^{\mathbf m(N)}|^2 \end{equation}
where for each $N=1,2,...$
\begin{equation}\label{multiplicity}\mathbf m(N):=|\{(n,r,u):i_{n,r}(u)=N\}|+|\{(m,s,v):l_{m,s}(v)=N\}|\end{equation}
$$=|\{(n,r,u):j_{n,r}(u)=N\}|+|\{(m,s,v):k_{m,s}(v)=N\}|$$
In this case
$$ E\left(\prod_{n,r} (\alpha_{i_{n,r}(1)}\alpha_{j_{n,r}(1)}^*...)(\prod_{m,s} (
\alpha_{k_{m,s}(1)}\alpha_{l_{m,s}(1)}^*...  )^*\right) = \prod_N \frac{\mathbf m(N)!}{(N\beta+1)...(N\beta+\mathbf m(N))} $$

\begin{notation}\label{mnotation} Given a function $i$ on a finite domain with values in nonnegative integers, let $\delta_{i}$ denote the density of the image of the function. We can then write
$$\mathbf m=\sum_{n,r}\delta_{i_{n,r}}+\sum_{m,s}\delta_{l_{m,s}} $$
Note that $\mathbf m(0)$ can be nonzero, so that $\mathbf m$ is a multi-index, with the indexing starting at $0$.
\end{notation}

\begin{theorem}\label{alphasum} With respect to the distribution (\ref{alphadistribution}), if $d=deg(p)=deg(q)$,
then $E(x^p(x^q)^*)$ is the sum over multi-indices
\begin{equation}\label{msum}\sum_{\mathbf m} C(p,q,\mathbf m)\prod_N \frac{\mathbf m(N)!}{(N\beta+1)...(N\beta+\mathbf m(N))} \end{equation}
where $C(p,q,\mathbf m)$ is the number of tuples of multi-indices $(i_{n,r},j_{n,r},k_{m,s},l_{m,s})$
satisfying the conditions: (1) for $n,m\ge 1$, $1\le r\le p(n)$ and $1\le s\le q(m)$; (2)
$$i_{n,r}(1)>j_{n,r}(1)>i_{n,r}(2)...>j_{n,r}(length(j_{n,r}))\ge 0 $$
$$ k_{m,s}(1)>l_{m,s}(1)>k_{m,s}(2)>...>l_{m,s}(length(l_{m,s}))\ge 0, $$
$$\sum_u (i_{n,r}(u)-j_{n,r}(u))=n,\qquad \sum_v (k_{m,s}(v)-l_{m,s}(v))=m $$
and (3)
$$\mathbf m=\sum_{n,r}\delta_{i_{n,r}}+\sum_{m,s}\delta_{l_{m,s}} = \sum_{n,r}\delta_{j_{n,r}}+\sum_{m,s}\delta_{k_{m,s}} $$

If $deg(p)\ne deg(q)$, then $E(x^p(x^q)^*)=0$.

\end{theorem}

The sum (\ref{msum}) is a holomorphic function of $\beta$ with apparent poles at the points $\beta=-k/N$,
$k\le \mathbf m(N)$, $N=1,...$, and a potential essential singularity at $\beta=0$. The following is essentially a
combinatorial statement. It is equivalent to (\ref{maintheorem2}).

\begin{theorem}\label{CNidentity} Suppose that $p,q$ are multi-indices with $d=deg(p)=deg(q)$.
Then the sum in Theorem \ref{alphasum},
$$\sum_{\mathbf m} C(p,q,\mathbf m)\prod_N \frac{\mathbf m(N)!}{(N\beta+1)...(N\beta+\mathbf m(N))}$$
equals the sum in Theorem \ref{bosonicsum} (which is a polynomial of degree $d$ in $\beta^{-1}$)
$$E_{\mathcal P_d}\left( f(p,L)f(q,L)\beta^{-|L|}\right)$$
where
$$f(p,L)= \sum_{J_{n,r}}\left(\begin{matrix}L\\(J_{n,r})\end{matrix}\right)$$
when $L\in \mathcal P_d$ can be written as $L=\sum_{n,r}J_{n,r}$ with $deg(J_{n,r})=n$, $r\le p(n)$,
and $f(p,L)=0$ otherwise.
\end{theorem}

\subsection{A Graphical Interpretation}

In this subsection we will describe a graphical interpretation of the combinatorial coefficient
$C(p,q,\mathbf m)$ in Theorem \ref{alphasum}. As in the statement of the Theorem, $d=deg(p)=deg(q)$
and $\mathbf m$ is a multi-index (where the indexing for $\mathbf m$ starts at $0$).

\begin{definition} A directed graph $G$ with set of vertices $V=\{i\in \mathbb Z_{\ge 0}: \mathbf{m}(i)>0 \}$ satisfies the   $\mathbf{m}$-\textbf{condition} if for each vertex $i\in V$, the number of ingoing edges equals the number of outgoing edges equals $\mathbf{m}(i)$. For such a graph $G$, we define the edge weight from $i$ to $j$ by $w_{ij}:=|j-i|$.
\end{definition}

We now describe a construction of a graph $G$ for given $\mathbf{m}$. We first construct an auxiliary graph $G_\mathbf{m}$ and then realize $G$ as a quotient of $G_\mathbf{m}$ by identifying some of it vertices.
Let $r=L(\mathbf{m})=\{ i: \mathbf{m}(i) >0 \}$, and define $G_\mathbf{m}$ to be the complete directed $r$-partite graph such that the partite corresponding to $i\in \mathbb Z_{\ge 0}$ has exactly $\mathbf{m}(i)$ vertices.



\begin{lemma}\label{lemma-cdecomposition}
Let $C_1, C_2, \dots$ be a disjoint collection of cycles of the directed graph $G_\mathbf{m}$ which includes all
vertices. Consider the subgraph $H=C_1\cup C_2 \cup \dots$ of $G_\mathbf{m}$ and identify the vertices in $H$ that correspond to the same integer. The graph $G$ constructed in this way will satisfies the $\mathbf{m}$-condition.
\end{lemma}

\begin{proof} By construction for each vertex in $H$ both the ingoing number of edges and outgoing number of edges are equal to one. Now, by definition, every integer $i$ has $\mathbf{m}(i)$ corresponding vertices in $G_\mathbf{m}$. Therefore, after the identification of vertices above any given integer $i$, both the ingoing and outgoing number of edges is equal to $\mathbf{m}(i)$.
\end{proof}

We now discuss the coloring of such a directed graph given budget constraints. For $i<j$ we call $e_{ij}$ (resp. $e_{ji}$) a positively oriented edge (resp. negatively oriented). We will use distinct sets of colors for each orientation.

\begin{definition} Consider a directed graph $G$ satisfying the $\mathbf{m}$-condition. Choose $|p|$ 'positive' colors and $|q|$ 'negative' colors, and partition the set of positive colors so that $p(u)$  of them have budget $u$. Similarly, partition the set of negative colors so that $q(u)$  of them have budget $u$ (If a color has budget $u$ it means that the total weight of edges colored with it have to add up to $u$). \\

A coloring of $G$ is \textbf{non-overlapping} if $i<j\leq k<l$ implies that the edges $e_{ik}$ and $e_{jl}$ have distinct colors (similarly the negatively oriented edges $e_{lj}$ and $e_{ki}$ should have distinct colors)

\end{definition}

Let's give an explicit example of the budget rule based on figure \ref{fig-graph}. If, for example, we have $q=(0,2,0,1,0,\dots)$ for negatively oriented edges. It implies that we have a total of 3 distinct colors, where 2 of them each have a budget constraint of $2$ and the other one has a budget constraint of $4$. Also, note that an edge $e_{ij}$ will allocate $|j-i|$ from the coloring budget based on the weight assigned to it.

\begin{figure}[H]
    \centering
    \includegraphics[width=350px]{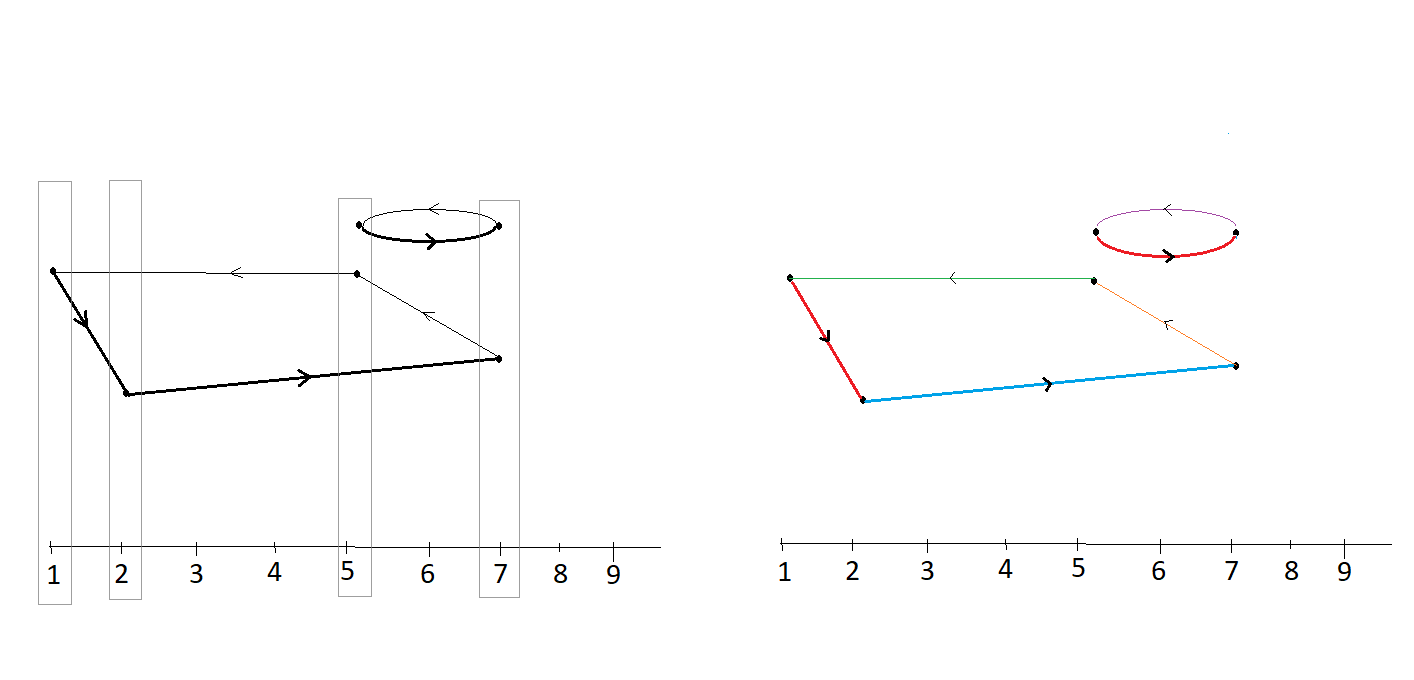}
    \caption{On the left, we see an example of a graph that after an identification along vertical strips will satisfy the $\textbf{m}$-condition, where  $\textbf{m}=(\mathbf m(0)=0,1,1,0,0,2,0,2,0,\dots)$ and bold edges are positively oriented. This can be seen as a cycle decomposition of the complete directed $r$-partite graph. On the right, we have a non-overlapping coloring with the budget constraint given by $p=(0,0,1,0,1,0,\dots)$ and $q=(0,2,0,1,0,\dots)$. }
    \label{fig-graph}
\end{figure}

\begin{theorem} Fix $p,q$ with $d=deg(p)=deg(q)$ and $\mathbf  m$. Let $G_1, \dots, G_n$ denote the set of all possible directed graphs $G$ satisfying $\mathbf{m}$-condition and let $c_j$ be the number of non-overlapping colorings of $G_j$ according to the color budget given by $p,q$. Then
$$C(p,q,\mathbf{m}) = c_1+\dots +c_n$$
Furthermore, the set of possible graphs satisfying $\mathbf{m}$-condition can be constructed as in Lemma \ref{lemma-cdecomposition}.
\end{theorem}

\begin{proof}
Here is how one translates the index families in theorem \ref{alphasum} into a graph with integer vertices. The indices defined by $i_{n,r}$ and $j_{n,r}$ (resp. $k_{m,s}$ and $l_{m,s}$) identify the positively oriented (resp. negatively oriented) edges in the graph, in the following way. For any pair $(n,r)$ and  $1\leq z \leq length(j_{n,r})$ we define a positively oriented edge from the vertex $v = i_{n,r}(z)$ to $w = j_{n,r}(z)$. The rule is similar for negatively oriented edges.
\end{proof}

We have the following dictionary for various combinatorial quantities, consistent with previous sections:

\begin{itemize}
    \item $L(\mathbf{m}) $: number of nodes in $G$.
    \item $|\mathbf{m}|$: number of edges in $G$.
    \item $\mathbf{m}(i)$: in-going and out-going degree of node $i$.

    \item $deg(p) $; total color budget used to color positively oriented edges.
    \item $deg(q) $: total color budget used to color negatively oriented edges.
    \item $|p|$: number of distinct colors used to color positively oriented edges.
    \item $|q|$: number of distinct colors used to color negatively oriented edges.
    \item $p(i)$: number of colors with budget $i$ for positively oriented edges.
    \item $q(i)$: number of colors with budget $i$ for negatively oriented edges.
\end{itemize}

\end{document}